\documentclass[reqno,a4paper, 11pt]{amsart}

\numberwithin{equation}{section}

\usepackage{amsfonts,amssymb,amsmath,amsthm}
\usepackage{enumerate}
\usepackage{bbm}
\usepackage{times}
\usepackage{amsfonts}
\usepackage{amssymb}
\usepackage{bbm}
\usepackage{times}
\usepackage{amsfonts}
\usepackage{amssymb}
\usepackage{bbm}
\usepackage{times}

\usepackage{enumerate}
\usepackage{amsmath}
\usepackage{amsthm}
\usepackage{color}

\usepackage [latin2]{inputenc}
\newtheorem{thm}{Theorem}[section]
\newtheorem{lem}[thm]{Lemma}

\newtheorem{prop}[thm]{Proposition}

\newcounter{other}            
\newtheorem{otherth}[other]{Theorem}              
\newtheorem{otherl}[other]{ Lemma}        

\newcommand{\aut}{{\rm Aut}(\D)}

\def\D{\mathbb{D}}
\def\T{\partial \D}
\def\C{\mathbb{C}}
\def\Q{\mathcal{Q}}
\def\B {\mathcal B}
\def \f{\frac}
\def \ind{ \int_\D }

\begin{document}

\title [A family of M\"obius invariant function spaces]
{A Carleson type measure and a family of M\"obius invariant function spaces}
\author{Guanlong Bao and Fangqin Ye}
\address{Guanlong Bao\\
    Shantou University\\
    Shantou, Guangdong 515063, China}
\email{glbao@stu.edu.cn}

\address{Fangqin Ye\\
   Shantou University\\
    Shantou 515063, Guangdong, China}
\email{fqye@stu.edu.cn}

\thanks{ The work was supported by NNSF of China (No. 12001352 and No. 12271328) and  Guangdong basic and applied basic research foundation (No. 2022A1515012117).}
\begin{abstract}
For $0<s<1$, let  $\{z_n\}$ be a sequence in the open unit disk such that $\sum_n (1-|z_n|^2)^s \delta_{z_n}$ is an $s$-Carleson measure.
 In this paper,  we consider the connections between this $s$-Carleson measure  and the theory of M\"obius invariant $F(p, p-2, s)$ spaces  by  the Volterra type operator, the reciprocal of  a  Blaschke product,  and second order complex differential equations having a prescribed zero sequence.
\end{abstract}

\maketitle

\section{Introduction}

 Let $\D$ be the open unit disk in the complex plane $\C$. A function $\varphi$ is said to be a M\"obius map if $\varphi$  is a  one-to-one analytic function that maps $\D$ onto itself. M\"obius maps  form a group with respect to
 the composition of mappings. This group is said to be   the M\"obius  group and denoted by $\aut$. For $a\in \D$, the following special M\"obius map defined by
 $$
 \sigma_a(z)=\frac{a-z}{1-\overline{a}z}, \ \ z\in \D,
 $$
 exchanges points $0$ and $a$. It is well known that every $\varphi$ in $\aut$  can be represented as $ \varphi=e^{i\theta} \sigma_a$ for some real number $\theta$ and some  $a$ in $\D$.   Clearly,   the M\"obius  group $\aut$ is homeomorphic to $\D\times \T$.

 A classical topic in complex analysis  is to investigate  the theory of  M\"obius invariant function spaces.  Let $H(\D)$ be the space of functions analytic in $\D$.  A space  $X$ contained in  $H(\D)$ is a M\"obius invariant function space if it is equipped with a semi-norm $\rho$ such that   $f\circ\varphi\in X$ and
$\rho(f\circ\varphi)\lesssim\rho(f)$ for all $f\in X$ and all $\varphi\in\aut$. In this case,  it is known  that there is  another semi-norm $\rho'$ on $X$ satisfying  that $\rho$ and $\rho'$ are equivalent,  and
$\rho'(f\circ\varphi)=\rho'(f)$ for all $f\in X$ and $\varphi\in\aut$. See \cite{AFP} for the theory on  M\"obius invariant function spaces.

In this paper, we consider  M\"obius invariant function spaces $F(p, p-2, s)$. For $0<p<\infty$ and  $0< s<\infty$, the space $F(p, p-2, s)$  is the set of  functions $f\in H(\D)$ for which
$$
\|f\|_{F(p, p-2, s)}^p=\sup_{a\in \D}\int_\D |f'(z)|^p(1-|z|^2)^{p-2}(1-|\sigma_a(z)|^2)^s dm(z)<\infty,
$$
where $dm(z)=1/\pi dxdy$,  $z=x+iy$.  It is known that  if $p+s\leq 1$, then $F(p, p-2, s)$ is trivial; that is, $F(p, p-2, s)$  consists only  of constant functions.  Since
$$
\|f\circ  \varphi-f(\varphi(0))\|_{F(p, p-2, s)}=\|f\|_{F(p, p-2, s)}
$$
for every $f\in F(p, p-2, s)$ and $\varphi \in \text{Aut}(\D)$, $F(p, p-2, s)$ is  M\"obius invariant. For $p=2$ and $s=1$, $F(p, p-2, s)$ is equal to  $BMOA$, the well-known  space of analytic functions in the Hardy space $H^1$ whose boundary values having bounded mean oscillation on $\T$ (cf. \cite{Bae, Gir}). For $p=2$, $F(p, p-2, s)$ is the well-studied space $\Q_s$ (cf. \cite{AXZ, Xi1, Xi2}). For $s>1$, all  $F(p, p-2, s)$ spaces are the same and equal to the Bloch space $\B$. Recall that $\B$ is the space of functions $f\in H(\D)$ such that
 $$
 \sup_{z\in \D}(1-|z|^2)|f'(z)|<\infty.
 $$
 $F(p, p-2, s)$ spaces are special cases of a general family of analytic function spaces $F(p, q, s)$ that were introduced by R. Zhao \cite{Zhao} and further investigated  in \cite{Rat}.  See \cite{Zhao2} for a recent survey on  $F(p, q, s)$ spaces.  We also  refer to
\cite{BP, BWY, QY} for some recent investigations on  $F(p, p-2, s)$ spaces.

Carleson type measures are key tools for the  modern function theory and operator theory. For  an arc $I$ of  $\T$ with arclength $|I|$,  the
 Carleson box  $S(I)$ is given by
$$
S(I)=\left\{r\zeta \in \D: 1-\frac{|I|}{2\pi}<r<1, \ \zeta\in I\right\}.
$$
  For $0<s<\infty$,  a nonnegative Borel measure $\mu$ on $\D$ is said to be an  $s$-Carleson measure  if $\mu(S(I))\lesssim |I|^s$ for all $I\subseteq\T$. When $s=1$, we obtain the classical Carleson measure
  that was introduced by L. Carleson \cite{C1, C2} to solve the problem of interpolation by functions in $H^\infty$ and  the corona problem. Here $H^\infty$ is the space of bounded analytic functions in $\D$ and we will write
  $\|f\|_\infty=\sup_{z\in \D}|f(z)|$ for $f\in H^\infty$.  It is also well known (cf. \cite{Xi1}) that $\mu$ is an  $s$-Carleson measure if and only if
  \begin{equation}\label{SCM}
  \sup_{a\in \D} \ind \left(\frac{1-|a|^2}{|1-\overline{a}z|^2}\right)^sd\mu(z)<\infty.
  \end{equation}

  For   $X\subseteq H^\infty$ and a sequence $\{z_n\}$ in $\D$,   $\{z_n\}$  is  called an interpolating sequence for $X$ if for every
bounded sequence $\{\zeta_n\}$ of complex numbers, there is   $f\in X$ with  $f(z_n) = \zeta_n$
for each   $n$.
By    L. Carleson \cite{C1},  $\{z_{n}\}$  is an interpolating sequence for $H^{\infty}$ if and only if $\{z_{n}\}$ is uniformly separated; that is, there exists a positive constant $\gamma$ satisfying
\begin{equation}\label{uniformlyseparatedd}
\inf_{m}\prod_{n\neq m} \left|\frac{z_m-z_n}{1-\overline{z_m}z_n}\right|\geq \gamma.
\end{equation}
A sequence $\{z_n\}$ in $\D$ is said to be  separated if
$\inf_{n \not=k}\rho(z_n, z_k)>0$,
where $\rho(z_n, z_k)=|\sigma_{z_n}(z_k)|$ is the pseudo-hyperbolic metric between $z_n$ and $z_k$.  For $0<r<1$ and $a\in \D$, denote by
$
\Delta(a, r)=\{z\in \D:  |\rho(a, z)|<r\}
$
the  pseudo-hyperbolic disk of center $a$ and radius $r$.  It is well known (cf. \cite{Ga}) that $\{z_n\}$ in $\D$ is uniformly separated if and only if
$\{z_n\}$ is separated and $\sum_n (1-|z_n|)\delta_{z_n}$ is a Carleson measure, where  $\delta_{z_n}$ is    the unit point-mass measure at $z_n$.  For a sequence $\{z_n\}$ in $\D$, it is known (cf. \cite{MS}) that
$\sum_n (1-|z_n|)\delta_{z_n}$ is a Carleson measure if and only if $\{z_n\}$ is finite unions of interpolation sequences for $H^\infty$. See \cite{DS, Ni, No} for the study of  $\{z_n\}$ in $\D$ satisfying that $\sum_n (1-|z_n|)\delta_{z_n}$ is a Carleson measure.

For $0<s<1$, let $\{z_n\}$ be  a sequence in $\D$  such that $\sum_n (1-|z_n|^2)^s \delta_{z_n}$ is an $s$-Carleson measure. This kind of Carleson type measures defined by sequences is very useful in the study of some theory of
$F(p, p-2, s)$ spaces.  F. P\'erez-Gonz\'alez and J. R\"atty\"a \cite{PR} used them to characterize inner functions in some $F(p, p-2, s)$ spaces. G. Bao and J. Pau \cite{BP} constructed some examples
based on these Carleson type measures to show that the multiplier from some $F(p, p-2, s)$ to another  is  only the zero function.  Ch. Yuan and C. Tong \cite{YT},  and R. Qian and F. Ye \cite{QY} used these  Carleson type measures to describe interpolating sequences for $H^\infty \cap F(p, p-2, s)$ with certain ranges of parameters $p$ and $s$. These  Carleson type measures are also used to investigate  solutions of second order complex differential equations having prescribed zeros (cf. \cite{Gr, Ye}). Recently, in \cite{BWY} the authors applied  these Carleson type measures  to construct Blaschke products for considering the proper inclusion relation associated with intersections and unions of some $F(p, p-2, s)$  spaces.

The aim of this paper is to consider  further the connections  between  the   $F(p, p-2, s)$  theory   and   sequence   $\{z_n\}$  in $\D$  such that $\sum_n (1-|z_n|^2)^s \delta_{z_n}$ is an $s$-Carleson measure. Our results involve  the  Volterra type operator from $H^\infty$ to $F(p, p-2, s)$, the  reciprocal of  a related Blaschke product, and solutions of second order complex differential equations having prescribed zeros $\{z_n\}$.

Throughout this paper,  we  write $a\lesssim b$ if there exists a positive constant $C$ such that $a\leq Cb$.
 The symbol $a\thickapprox b$ means  $a\lesssim b\lesssim a$.

\section{Volterra type operators from $H^\infty$ to $F(p, p-2, s)$}

In this section,  we give a  relation between  the  Volterra type operator  from $H^\infty$ to $F(p, p-2, s)$, and   $\{a_k\}$  in $\D$  satisfying  that $\sum_{k=1}^\infty  (1-|a_k|^2)^s \delta_{a_k}$ is an $s$-Carleson measure.  Our proof also yields that the range of the Ces\`aro operator acting on $H^\infty$ is contained in every non-trivial $F(p, p-2, s)$ space, which strengthens  some  previous results  from the literature.

The Volterra type operator $J_g$ with symbol $g\in H(\D)$ defined on $H(\D)$ by
$$
J_g f(z)=\int_0^z f(w)g'(w)dw, \ \ z\in \D, \ \ f\in H(\D).
$$
For the companion operator
$$
I_g f(z)=\int_0^z f'(w)g(w)dw, \ \ z\in \D, \ \ f\in H(\D),
$$
and the multiplier operator
$$
M_g f(z)=g(z)f(z), \ \ z\in \D, \ \ f\in H(\D),
$$
it is clear that
\begin{equation}\label{20}
M_g f(z)=f(0)g(0)+J_g f(z)+I_g f(z).
\end{equation}
Ch. Pommerenke \cite{P} first studied  $J_g$ and  proved that  $J_g$ is a bounded operator on the Hardy space $H^2$ if and only if $g\in BMOA$. Later, the operator  $J_g$ was  studied systematically in \cite{AS}. Operators $J_g$ and $I_g$ have attracted a lot of interest. Here we mention the paper \cite{PZ} studying  these operators related to $F(p, q, s)$ spaces.

An inner function $f$ is an $H^\infty$ function whose boundary values satisfy $|f(e^{i\theta})|=1$ almost everywhere.
A sequence $\{a_k\}$  in $\D$ is called  a Blaschke sequence if
$$
\sum_k (1-|a_k|)<\infty.
$$
This condition  ensures  the convergence of the  Blaschke product:
$$
B(z)=\prod_{k=1}^\infty \f{|a_k|}{a_k}\f{a_k-z}{1-\overline{a_k}z}.
$$
All   Blaschke products are  inner functions. If the zero sequence of a Blaschke product $B$ is uniformly separated, then $B$ is said to be an
interpolating Blaschke product.  A Blaschke product is called a Carleson-Newman Blaschke product if it can
be expressed as a product of finitely many interpolating Blaschke products. Recall  that a Blaschke product associated with  a sequence
$\{a_k\}$ in $\D$ is a Carleson-Newman Blaschke product if and only if $\sum_k (1-|a_k|^2)\delta_{a_k}$ is a 1-Carleson measure.

The following result  is Theorem 1.4 in \cite{PR} characterizing inner functions in some $F(p, p-2, s)$ spaces.

\begin{otherth} \label{1F-Inner}
Let $0<s<1$ and $p>\max\{s, 1-s\}$. Then an inner function belongs to $F(p, p-2, s)$ if and only if it is a Blaschke product associated with a sequence $\{z_k\}_{k=1}^\infty$ in $\D$ which satisfies that
$\sum_k (1-|z_k|)^s \delta_{z_k}$ is an $s$-Carleson measure.
\end{otherth}

We give the following  conclusion  first.
\begin{thm}\label{intergral operat}
Suppose $g\in H(\D)$, $0<p<\infty$ and $0<s<\infty$ satisfying  $p+s>1$. Then the following statements hold:
\begin{enumerate}
  \item [(a)] if   $s>1$, or $s=1$ and $p\geq 2$,  then $I_g$ is a bounded operator from $H^\infty$ to $F(p, p-2, s)$ if and only if $g\in H^\infty$;
  \item [(b)] if   $0<s<1$, or $s=1$ and $0<p<2$,  then $I_g$ is a bounded operator from $H^\infty$ to $F(p, p-2, s)$ if and only if $g=0$;
  \item [(c)] $J_g$ is a bounded operator from $H^\infty$ to $F(p, p-2, s)$ if and only if $g\in F(p, p-2, s)$;
  \item [(d)] if   $s>1$, or $s=1$ and $p\geq 2$,  then $M_g$ is a bounded operator from $H^\infty$ to $F(p, p-2, s)$ if and only if  $g\in H^\infty$;
  \item [(e)] if   $0<s<1$, or $s=1$ and $0<p<2$, then $M_g$ is a bounded operator from $H^\infty$ to $F(p, p-2, s)$ if and only if  $g=0$.
\end{enumerate}
\end{thm}
\begin{proof}
(a)\   Note that $F(p_1, p_1-2, s_1)\subseteq F(p_2, p_2-2, s_2)$ for all possible $0<p_1\leq p_2$ and  $0<s_1\leq s_2$. It is also known that  $H^\infty\subseteq BMOA\subseteq \B$,  $BMOA=F(2, 0, 1)$, and  $F(p, p-2, s)=\B$ for $s>1$. Now consider    $s>1$, or $s=1$ and $p\geq 2$. Then  $H^\infty\subseteq F(p, p-2, s)$ for such $p$ and $s$. Let  $g\in H^\infty$. For any $f\in H^\infty$, then $f$ also belongs to $F(p, p-2, s)$. Thus
$$
\|I_g f\|^p_{F(p, p-2, s)}\leq \|g\|_\infty^p \| f\|^p_{F(p, p-2, s)}\lesssim \|g\|_\infty^p\|f\|_\infty^p,
$$
 which yields the boundedness of  $I_g$  from $H^\infty$ to $F(p, p-2, s)$. On the other hand, suppose $I_g$ is  bounded  from $H^\infty$ to $F(p, p-2, s)$. Note that  $\sup_{b\in \D}\|\sigma_b\|_\infty \leq 1$.  Then
\begin{align*}
\infty & > \sup_{b\in \D} \|I_g \sigma_b\|^p_{F(p, p-2, s)}\\
& \geq \sup_{b\in \D} \sup_{a\in \D} \int_{\Delta(b, 1/2)} \frac{(1-|b|^2)^p}{|1-\overline{b}z|^{2p}}|g(z)|^p (1-|z|^2)^{p-2}(1-|\sigma_a(z)|^2)^s dm(z).
\end{align*}
It is well known (cf.  \cite[p. 69]{Zhu} and \cite[Lemma 4.30]{Zhu})   that
$$
1-|z|\thickapprox1-|b|\thickapprox |1-\overline{b} z|
$$
for all $z\in \Delta(b, 1/2)$, and
$
 |1-\overline{z} a|\thickapprox  |1-\overline{b} a|
$
for all $z\in \Delta(b, 1/2)$ and all $a\in \D$. Also the area of $\Delta(b, 1/2)$ is comparable with $(1-|b|)^2$.  Combining these facts and the  subharmonicity  of $|g|^p$, we obtain
$$
1\gtrsim   \sup_{a\in \D} |g(b)|^p (1-|\sigma_a(b)|^2)^s=  |g(b)|^p
$$
for all $b\in \D$.  Thus $g\in H^\infty$.

(b)\  First, consider  $0<s<1$. Suppose  $I_g$ is a bounded operator from $H^\infty$ to $F(p, p-2, s)$. By the proof of (a),  we get $g\in H^\infty$.  For every $e^{i\theta}\in \T$, from  Lemma 3.3 in \cite{BP}, there exists a Blaschke sequence $\{a_k\}_{k=1}^\infty$ in $\D$ such that
$e^{i\theta}$ is the unique accumulation point of $\{a_k\}_{k=1}^\infty$,
\begin{equation}\label{21}
\sup_{a\in \D} \sum_{k=1}^\infty (1-|\sigma_a(a_k)|^2)^{\frac{1+s}{2}}<\infty,
\end{equation}
and
\begin{equation}\label{22}
\sup_{a\in \D} \sum_{k=1}^\infty (1-|\sigma_a(a_k)|^2)^s=+\infty.
\end{equation}
Condition (\ref{21}) implies that $\sum_{k=1}^\infty  (1-|a_k|)\delta_{a_k}$ is a Carleson measure. Hence  $\{a_k\}_{k=1}^\infty$ is finite unions of interpolation sequences for $H^\infty$.  Then there exists a positive integer $m$ and positive real numbers $\gamma_i$, $i=1, 2, \cdots, m$, such that
$$
\{a_k\}=\bigcup_{1\leq i \leq m} \{a_{ik}\}, \ \ \text{and} \ \ \inf_\ell \prod_{j\not= \ell} \rho(a_{ij}, a_{il})\geq \gamma_i, \ \ i=1, 2, \cdots, m.
$$
Because of  (\ref{22}), among these sequences $\{a_{ik}\}_{k=1}^\infty$, $i=1, 2, \cdots, m$, there exists a sequence $\{a_{i_0k}\}_{k=1}^\infty$ such that
\begin{equation}\label{23}
\sup_{a\in \D} \sum_{k=1}^\infty (1-|\sigma_a(a_{i_0k})|^2)^s=+\infty.
\end{equation}
Of course, $e^{i\theta}$ is also the unique accumulation point of $\{a_{i_0k}\}_{k=1}^\infty$. Denote by $B$ the Blaschke product associated with the sequence $\{a_{i_0k}\}_{k=1}^\infty$. By a well-known fact of interpolating sequences for $H^\infty$ (cf. \cite[p. 681]{GPV}), we know
$$
\bigcup_{k=1}^{\infty}\Delta\left(a_{i_0k},\frac{\gamma_{i_0}}{4}\right)\subseteq
\left\{z\in\mathbb{D}:(1-|z|)|B'(z)|\geq \frac{\gamma_{i_0}(1-\gamma_{i_0})}{8}\right\}.
$$
Clearly, pseudo-hyperbolic disks $\Delta\left(a_{i_0k},\frac{\gamma_{i_0}}{4}\right)$ are pairwise disjoint.
Consequently,
\begin{align*}
\infty  >&  \|I_g B\|_{F(p, p-2, s)}^p \\
=& \sup_{a\in \D} \int_\D |B'(z)|^p |g(z)|^p (1-|z|^2)^{p-2}(1-|\sigma_a(z)|^2)^s dm(z)\\
\geq & \sup_{a\in \D} \int_{\left\{z\in\mathbb{D}:(1-|z|)|B'(z)|\geq \frac{\gamma_{i_0}(1-\gamma_{i_0})}{8}\right\}} |B'(z)|^p |g(z)|^p\\
& \times  (1-|z|^2)^{p-2}(1-|\sigma_a(z)|^2)^s dm(z)\\
 \gtrsim  &  \sup_{a\in \D}  \sum_{k=1}^\infty \int_{\Delta\left(a_{i_0k},\frac{\gamma_{i_0}}{4}\right)} |g(z)|^p \frac{(1-|\sigma_a(z)|^2)^s}{(1-|z|^2)^2}dm(z)\\
 \gtrsim  &  \sup_{a\in \D}  \sum_{k=1}^\infty |g(a_{i_0k})|^p (1-|\sigma_a(a_{i_0k})|^2)^s.
\end{align*}
Combining this with (\ref{23}), we get $|g(e^{i\theta})|=\lim_{k\rightarrow \infty}|g(a_{i_0k})|=0$. Due to the arbitrariness of $e^{i\theta}$ and the maximum modulus principle, we get $g\equiv0$.
Conversely, if $g\equiv0$, it is clear that $I_g$ is bounded from $H^\infty$ to $F(p, p-2, s)$.

Second, consider the case of $s=1$ and $0<p<2$. Also,  $g\equiv0$ implies clearly the boundedness of $I_g$ from $H^\infty$ to $F(p, p-2, s)$. Next, we  follow a method from  \cite[p. 177]{CGP}.   Let $I_g$  be bounded from $H^\infty$ to $F(p, p-2, s)$. As proved in part (a), $g\in H^\infty$. Suppose $g\not \equiv 0$. Then there exists a set $E$ in $[0, 2\pi]$ with positive Lebesgue measure such  that  $\lim_{r\rightarrow 1}g(re^{i\theta})\not=0$ for every $\theta\in E$.
 For $0<p<2$,  it is known from \cite[Theorem 1]{Gir1} that there exists $f\in H^\infty$ such that
\begin{equation}\label{24}
\int_0^1 |f'(re^{i\theta})|^p (1-r)^{p-1}dr=+\infty,
\end{equation}
for almost every $\theta\in [0, 2\pi]$; that is, (\ref{24}) holds when $\theta$ belongs to  some set $F$ in $[0, 2\pi]$ and  the  Lebesgue measure of $F$ is  $2\pi$. Thus for any $\theta \in E\cap F$, there exists an $r(\theta)$ in $(0, 1)$ satisfying that $\inf_{r(\theta)<r<1}|g(re^{i\theta})|>0$ and hence
\begin{align*}
& \int_0^1 |f'(re^{i\theta})|^p |g(re^{i\theta})|^p(1-r)^{p-1}dr\\
\geq &\inf_{r(\theta)<r<1}|g(re^{i\theta})| \int_{r(\theta)}^1 |f'(re^{i\theta})|^p (1-r)^{p-1}dr\\
=& +\infty.
\end{align*}
 Clearly,  the Lebesgue measure of $E\cap F$ is also positive.  We get
\begin{equation}\label{25}
\int_\D  |f'(z)|^p |g(z)|^p (1-|z|)^{p-1}dm(z)=+\infty.
\end{equation}
But the boundedness of $I_g$ from $H^\infty$ to $F(p, p-2, s)$ gives
\begin{align*}
\infty>&\|I_g f\|_{F(p, p-2, 1)}^p  \geq \int_\D  |f'(z)|^p |g(z)|^p (1-|z|^2)^{p-1}dm(z),
\end{align*}
which contradicts (\ref{25}). Hence $g\equiv0$.

(c)\  Suppose  $J_g$ is  bounded  from $H^\infty$ to $F(p, p-2, s)$. Set  $f(z)\equiv 1$.  Then $J_g f\in F(p, p-2, s)$.  Since   $g(z)=J_g f(z)+g(0)$,  we get   $g\in F(p, p-2, s)$.  Conversely, let $g\in F(p, p-2, s)$. For $h\in H^\infty$, it is clear that $\|J_g h\|_{F(p, p-2, s)}^p \leq \|h\|_\infty^p \|g\|_{F(p, p-2, s)}^p $, which gives the bondedness of  $J_g$  from $H^\infty$ to $F(p, p-2, s)$.

(d)\ Let  $M_g$ be a bounded operator from $H^\infty$ to $F(p, p-2, s)$. Because  constant functions are in $H^\infty$, we get $g\in F(p, p-2, s)$. By (c), $J_g$ is bounded from $H^\infty$ to $F(p, p-2, s)$. Because of (\ref{20}), $I_g$ is also bounded from $H^\infty$ to $F(p, p-2, s)$. From (a), we get  $g\in H^\infty$. Conversely,  let $g\in H^\infty$. Since  $s>1$, or $s=1$ and $p\geq 2$, $g$ also belongs to $F(p, p-2, s)$. By (a), (c) and (\ref{20}), we get the boundedness of $M_g$  from $H^\infty$ to $F(p, p-2, s)$.

The proof of (e) is similar to the proof of (d),  so we omit it.
\end{proof}

The following result gives a connection between the Volterra type operator  from $H^\infty$ to $F(p, p-2, s)$ and the Carleson type measures we considered.

\begin{thm}\label{1main}
Suppose $B$ is a Blaschke product with zeros $\{a_k\}_{k=1}^\infty$. Let $0<s<1$ and $p>\max\{s, 1-s\}$. Then the following conditions are equivalent:
\begin{itemize}
  \item [(a)] $\sum_{k=1}^\infty  (1-|a_k|^2)^s \delta_{a_k}$ is an $s$-Carleson measure;
  \item [(b)]   $J_B$ is a bounded operator from $H^\infty$ to $F(p, p-2, s)$.
\end{itemize}
\end{thm}
\begin{proof}
This result is clearly from Theorem \ref{1F-Inner} and Theorem \ref{intergral operat}.
\end{proof}

If $g(z)=-\log(1-z)$, then $J_g$ is the modified Ces\`aro operator $\widetilde{\mathcal C}$, namely,
$$
\widetilde{\mathcal C}  f(z)=\int_0^z \frac{f(w)}{1-w}dw, \ \ z\in \D, \ \ f\in H(\D).
$$
Note that the Ces\`aro operator $\mathcal C$ is defined by
$$
\mathcal C  f(z)=\frac{1}{z}\int_0^z \frac{f(w)}{1-w} dw, \ \ z\in \D, \ \ f\in H(\D).
$$
 N. Danikas and A. Siskakis \cite{DaSi} showed   that
$\mathcal{C }(H^\infty)\nsubseteq H^\infty$ but $\mathcal{C }(H^\infty)\subseteq BMOA$. M. Ess\'en and J. Xiao \cite{EX} obtained  that $\mathcal{C }(H^\infty)\subseteq \Q_s$ \ for \  $0<s<1$. Note that $F(p, p-2, s)\subsetneqq \Q_s$ when $0<p<2$ and $ \max\{1-p, 0\}<s\leq 1$. We
obtain smaller M\"obius invariant function  spaces closing to   $\mathcal{C }(H^\infty)$ as follows.

\begin{thm}\label{cesaro operator}
Let $0<p<\infty$ and $0<s<\infty$ with $p+s>1$. Then $\mathcal{C }(H^\infty)\subseteq F(p, p-2, s)$.
\end{thm}
\begin{proof}
It is known (cf. \cite{PZ}) that  the function $g(z)=-\log(1-z)$ belongs to all nontrivial $F(p, p-2, s)$ spaces. By Theorem \ref{intergral operat}, $J_g$ is a bounded operator from $H^\infty$ to $F(p, p-2, s)$; that is,
$\widetilde{\mathcal C} (H^\infty)\subseteq F(p, p-2, s)$. Hence $\mathcal{C }(H^\infty)\subseteq F(p, p-2, s)$.
\end{proof}
One can refer to \cite{BSW, BWY1, GGM} for more results  on  the range of the Ces\`aro operator or Ces\`aro-like operators acting on $H^\infty$.

\section{$s$-Carleson measure $\sum_{n=1}^{\infty} (1-|z_n|^2)^s\delta_{z_n}$ via the  reciprocal of  a  Blaschke product in $F(p, p-2, s)$}

Let $\{z_n\}_{n=1}^\infty$ be a separated  Blaschke sequence in $\D$ and let $B$ be the Blaschke product associated with $\{z_n\}_{n=1}^\infty$. Suppose $0<p<2$.
By Theorem C in \cite{No},  $\sum_{n=1}^{\infty} (1-|z_n|^2)\delta_{z_n}$ is a Carleson measure if and only if
\begin{equation}\label{N condition1}
\sup_{\varphi \in \aut}\int_\D \frac{1}{|B(\varphi(z))|^p}dm(z)<\infty.
\end{equation}
Clearly, (\ref{N condition1}) is equivalent to
\begin{equation}\label{N condition}
\sup_{\varphi \in \aut}\int_\D \left(\frac{1}{|B(\varphi(z))|}-1\right)^pdm(z)<\infty.
\end{equation}
In this section, for $0<s<1$ and  a separated  Blaschke sequence $\{z_n\}_{n=1}^\infty$,    based on the description of Blaschke products in $F(p, p-2, s)$, we give a characterization of $s$-Carleson measure $\sum_{n=1}^{\infty} (1-|z_n|^2)^s\delta_{z_n}$ in terms of the reciprocal of the Blaschke product with zeros $\{z_n\}$. Indeed, this Blaschke product belongs to some $F(p, p-2, s)$ spaces.

 By (\ref{SCM}), for $s>0$,   $\sum_{n=1}^{\infty} (1-|z_n|^2)^s\delta_{z_n}$ is an  $s$-Carleson measure if and only if
\begin{equation}\label{31}
\sup_{\varphi\in \aut}\sum_{n=1}^{\infty} \left(1-|\varphi(z_n)|^2\right)^s <\infty.
\end{equation}
Let $\phi\in \aut$.  From   (\ref{31}),  if $\sum_{n=1}^{\infty} (1-|z_n|^2)^s\delta_{z_n}$ is an  $s$-Carleson measure, then $\sum_{n=1}^{\infty} (1-|\phi(z_n)|^2)^s\delta_{\phi(z_n)}$ is also  an  $s$-Carleson measure.
In this sense,   $s$-Carleson measure $\sum_{n=1}^{\infty} (1-|z_n|^2)^s\delta_{z_n}$ is   M\"obius invariant. Hence the equivalent characterization we will give  also  shows  this invariance.

We begin with the following auxiliary result.

\begin{lem}\label{1  auxiliary}
Let $0<p<\infty$, $0\leq q<\infty$,  and $0<s<1$ such that  $p+s>1$.  Suppose $B$ is a Blaschke product associated with  $\{z_n\}_{n=1}^\infty$ in $\D$. If
 $$
  \sup_{\varphi \in \aut} \int_\D \frac{\left(1-|B(\varphi(z))|\right)^p}{|B(\varphi(z))|^q} (1-|z|^2)^{s-2}dm(z)<\infty,
  $$
  then $\sum_{n=1}^\infty  (1-|z_n|^2)^s \delta_{z_n}$ is an $s$-Carleson measure.
\end{lem}
\begin{proof}
 The Schwarz-Pick Lemma gives
$$(1-|z|^2)|B'(z)|\leq 1-|B(z)|^2$$
for all $z\in \D$. Combining this with the change of variables, we deduce
\begin{align*}
& \sup_{a\in \D}\int_\D  |B'(z)|^p (1-|z|^2)^{p-2}(1-|\sigma_a(z)|^2)^s dm(z)\\
\lesssim &  \sup_{a\in \D} \ind  \frac{\left(1-|B(z)|\right)^p}{|B(z)|^q} \frac{(1-|\sigma_a(z)|^2)^s}{(1-|z|^2)^2} dm(z)\\
\lesssim & \sup_{\varphi \in \aut} \int_\D \frac{\left(1-|B(\varphi(z))|\right)^p}{|B(\varphi(z))|^q} (1-|z|^2)^{s-2}dm(z)\\
 <& \infty.
\end{align*}
Hence $B\in F(p, p-2, s)$. Note that  $0<p<\infty$ and $0<s<1$ with $p+s>1$.  It follows from Theorem 4.3 in \cite{PR}  that  $\sum_{n=1}^\infty  (1-|z_n|^2)^s \delta_{z_n}$ is an $s$-Carleson measure.
\end{proof}

The following useful lemma is from \cite{K}.
\begin{otherl}   \label{B away from 0}
If $B(z)$ is the Blaschke product associated with  a sequence $\{z_n\}$ satisfying  condition (\ref{uniformlyseparatedd}) for a given $\gamma$, and if
$\epsilon>0$ is given, there exists a constant $\gamma_0$ depending only on $\gamma$ and $\epsilon$ such that
$|B(z)|\geq \gamma_0$ whenever $\rho(z, z_n)\geq \epsilon$ for all $n$.
\end{otherl}

Lemma \ref{|B|} below is also well-known; see \cite[p. 100]{Ye}  for  a brief proof of it.

\begin{otherl} \label{|B|}
Let  $B(z)$ be the Blaschke product  associated with  a sequence $\{z_n\}$ satisfying  condition (\ref{uniformlyseparatedd}) for a given $\gamma$.  Then there exists a positive constant $C$ depending only on
$\gamma$ such that
$$
|B(z)|\geq C \rho(z, z_k)
$$
for all  $z\in \Delta(z_k, \delta/4)$ and for every $k$.
\end{otherl}

The following conclusion   is  Corollary 2.4 in \cite{PR}.
\begin{otherl} \label{2F-Inner}
Let $S$ be an inner function and let $1\leq p<\infty$, $-2<q<\infty$, and $0\leq s, p^*<\infty$ such that $p>q+s+1>0$. Then, for any analytic function $f$ in $\D$ and $a\in \D$, the following quantities are comparable:
\begin{itemize}
  \item [(a)] $$\int_\D |f(z)|^{p^*}(1-|S(z)|^2)^p (1-|z|^2)^{q-p}(1-|\sigma_a(z)|^2)^s dm(z); $$
  \item [(b)] $$\int_\D |f(z)|^{p^*}|S'(z)|^p (1-|z|^2)^{q}(1-|\sigma_a(z)|^2)^s dm(z). $$
\end{itemize}
\end{otherl}

Note that if  $0<s<1$ and $\sum_{n=1}^\infty  (1-|z_n|^2)^s \delta_{z_n}$ is an $s$-Carleson measure,  then $\{z_n\}_{n=1}^\infty$ is finite unions of uniformly separated sequences.
Using Lemma \ref{B away from 0},  Lemma \ref{|B|} and Lemma \ref{2F-Inner}, we prove the following result.

\begin{lem}\label{2  auxiliary}
Suppose $0<s<1$ and  $\{z_n\}_{n=1}^\infty$ is a sequence in $\D$ such that $\sum_{n=1}^\infty  (1-|z_n|^2)^s \delta_{z_n}$ is an $s$-Carleson measure.   Let  $B$ be the  Blaschke product associated with  $\{z_n\}_{n=1}^\infty$. Assume   $1\leq p <\infty$  and $0\leq q<2/k$, where $k$ is the number of unions of uniformly separated sequences as  which   $\{z_n\}_{n=1}^\infty$ can be written. Then
$$
  \sup_{\varphi \in \aut} \int_\D \frac{\left(1-|B(\varphi(z))|\right)^p}{|B(\varphi(z))|^q} (1-|z|^2)^{s-2}dm(z)<\infty.
  $$
\end{lem}
\begin{proof}  Since  $\{z_n\}_{n=1}^\infty$ is $k$ unions of uniformly separated sequences,  there exist  positive real numbers $\gamma_i$, $i=1, 2, \cdots, k$, such that
$$
\{z_n\}_{n=1}^\infty=\bigcup_{1\leq i \leq k} \{z_{in}\}_{n=1}^\infty,
$$
and
$$
\inf_\ell \prod_{j\not= \ell} \rho(z_{ij}, z_{il})\geq \gamma_i, \ \ i=1, 2, \cdots, k.
$$
Denote by $B_i$ the Blaschke product associated with $\{z_{in}\}_{n=1}^\infty$.

By the change of variables, we get
  \begin{align} \label{equivalent b}
&  \sup_{\varphi \in \aut} \int_\D \frac{\left(1-|B(\varphi(z))|\right)^p}{|B(\varphi(z))|^q} (1-|z|^2)^{s-2}dm(z) \nonumber \\
=& \sup_{a\in \D} \ind \frac{\left(1-|B(z)|\right)^p}{|B(z)|^q} \frac{(1-|\sigma_a(z)|^2)^s}{(1-|z|^2)^2} dm(z).
\end{align}
From   a version of H\"older's inequality that handles the product of more than two functions (cf. \cite[Theorem 3.4]{Zhu}), we get
\begin{align} \label{060301}
&  \sup_{a\in \D} \ind \frac{\left(1-|B(z)|\right)^p}{|B(z)|^q} \frac{(1-|\sigma_a(z)|^2)^s}{(1-|z|^2)^2} dm(z)  \nonumber \\
\leq &  \sup_{a\in \D}  \prod_{i=1}^k  \left(\ind  \frac{1}{|B_i(z)|^{qx_i}}   \frac{\left(1-|B(z)|\right)^p(1-|\sigma_a(z)|^2)^s}{(1-|z|^2)^2} dm(z)\right)^{1/x_i},
\end{align}
where $x_1$, $\cdots$, $x_k$ are positive numbers such that
\begin{equation}\label{0603bu}
\frac{1}{x_1}+\cdots+\frac{1}{x_k}=1.
\end{equation}

By Lemma \ref{B away from 0}, for fixed $i=1, 2, \cdots, k$,  there exists a positive real number $\alpha_i$ depending only on $\gamma_i$ such that $|B_i(z)|\geq \alpha_i$ whenever $\rho(z, z_{in})\geq \gamma_i/4$ for all $n$. Write
$$
\Omega=\D \setminus \left(\bigcup_{1\leq i \leq k} \bigcup_{n=1}^\infty \Delta(z_{in}, \frac{\gamma_i}{4}) \right).
$$
Note that $q \geq 0$.
Then
\begin{align} \label{060302}
& \sup_{a\in \D}  \prod_{i=1}^k  \left(\int_\Omega  \frac{1}{|B_i(z)|^{qx_i}}   \frac{\left(1-|B(z)|\right)^p(1-|\sigma_a(z)|^2)^s}{(1-|z|^2)^2} dm(z)\right)^{1/x_i} \nonumber \\
\leq &   \prod_{i=1}^k  \frac{1}{\alpha_i^q} \sup_{a\in \D}  \left(\int_\D    \frac{\left(1-|B(z)|\right)^p(1-|\sigma_a(z)|^2)^s}{(1-|z|^2)^2} dm(z)\right)^{1/x_i}.
\end{align}
Bear in mind that  $0<s<1$ and  $1\leq p <\infty$.  It follows from Lemma \ref{2F-Inner} and Theorem \ref{1F-Inner} that
\begin{equation} \label{060303}
 \sup_{a\in \D}  \int_\D    \frac{\left(1-|B(z)|\right)^p(1-|\sigma_a(z)|^2)^s}{(1-|z|^2)^2} dm(z)\thickapprox \|B\|_{F(p, p-2, s)}^p<\infty.
\end{equation}
Thus, (\ref{060302}) and (\ref{060303}) yield
\begin{equation} \label{060304}
\sup_{a\in \D}  \prod_{i=1}^k  \left(\int_\Omega  \frac{1}{|B_i(z)|^{qx_i}}   \frac{\left(1-|B(z)|\right)^p(1-|\sigma_a(z)|^2)^s}{(1-|z|^2)^2} dm(z)\right)^{1/x_i}<\infty.
\end{equation}

On the other hand, by Lemma \ref{|B|}, we deduce
\begin{align*}
& \sup_{a\in \D}  \prod_{i=1}^k  \left(\int_{\D \setminus \Omega}  \frac{1}{|B_i(z)|^{qx_i}}   \frac{\left(1-|B(z)|\right)^p(1-|\sigma_a(z)|^2)^s}{(1-|z|^2)^2} dm(z)\right)^{1/x_i} \nonumber \\
\lesssim &    \sup_{a\in \D}  \prod_{i=1}^k   \left( \sum_{i=1}^k  \sum_{n=1}^\infty \int_{\Delta(z_{in}, \frac{\gamma_i}{4}) }  \frac{\left(1-|B(z)|\right)^p(1-|\sigma_a(z)|^2)^s}{(\rho(z, z_{in}))^{qx_i}(1-|z|^2)^2} dm(z) \right)^{1/x_i} \nonumber \\
\lesssim &    \sup_{a\in \D}  \prod_{i=1}^k   \left( \sum_{i=1}^k  \sum_{n=1}^\infty \frac{(1-|\sigma_a(z_{in})|^2)^s}{(1-|z_{in}|)^2}  \int_{\Delta(z_{in}, \frac{\gamma_i}{4})}  (\rho(z, z_{in}))^{-qx_i} dm(z)  \right)^{1/x_i}.
\end{align*}
If every $x_i$ satisfies  $qx_i<2$, then a change of variables yields
\begin{align*}
&\int_{\Delta(z_{in}, \frac{\gamma_i}{4})}  (\rho(z, z_{in}))^{-qx_i} dm(z)\thickapprox (1-|z_{in}|)^2 \int_0^{\frac{\gamma_i}{4}} r^{1-qx_i}dr\thickapprox (1-|z_{in}|)^2.
\end{align*}
Condition (\ref{0603bu}) and  $qx_i<2$ for every $x_i$ yield the assumed  condition  $q<2/k$. In fact, taking  $x_1=\cdots=x_k=k$ is   enough for this  proof. Combining these with (\ref{31}), we get
\begin{align} \label{060305}
&
\sup_{a\in \D}  \prod_{i=1}^k  \left(\int_{\D \setminus \Omega}  \frac{1}{|B_i(z)|^{qx_i}}   \frac{\left(1-|B(z)|\right)^p(1-|\sigma_a(z)|^2)^s}{(1-|z|^2)^2} dm(z)\right)^{1/x_i} \nonumber \\
\lesssim & \sup_{a\in \D}  \prod_{i=1}^k  \left(  \sum_{n=1}^\infty (1-|\sigma_a(z_{n})|^2)^s  \right)^{1/x_i}<\infty.
\end{align}
Joining (\ref{equivalent b}), (\ref{060301}), (\ref{060304}) and (\ref{060305}),  we get the desired result.
\end{proof}

Now we state  the main result of this section as follows.

\begin{thm}\label{2main}
Let $0<s<1$ and $1\leq p< 2$.  Suppose  $\{z_n\}_{n=1}^\infty$ is a separated  Blaschke sequence in $\D$ and  $B$ is  the Blaschke product associated with $\{z_n\}_{n=1}^\infty$.   Then the following conditions are equivalent:
\begin{itemize}
  \item [(a)]  $\sum_{n=1}^\infty  (1-|z_n|^2)^s \delta_{z_n}$ is an $s$-Carleson measure;
  \item [(b)] $$
  \sup_{\varphi \in \aut} \int_\D \left(\frac{1}{|B(\varphi(z))|}-1\right)^p (1-|z|^2)^{s-2}dm(z)<\infty.
  $$
\end{itemize}
\end{thm}
\begin{proof}
Note that if   $\{z_n\}_{n=1}^\infty$ is a separated  Blaschke sequence such that  $\sum_{n=1}^\infty  (1-|z_n|^2)^s \delta_{z_n}$ is an $s$-Carleson measure for  $0<s<1$, then $\{z_n\}_{n=1}^\infty$ is a  uniformly separated sequence.
 Set  $q=p$ in Lemma \ref{1  auxiliary} and Lemma \ref{2  auxiliary}.  The conclusion follows.
\end{proof}

Taking  $s=1$ in (b) of Theorem \ref{2main}, we do not get C. Nolder's condition (\ref{N condition}).
 In the proof of Theorem \ref{2main}, the characterization of Blaschke products in $F(p, p-2, s)$  plays an important role, where the condition $0<s<1$ was used.

\section{$s$-Carleson measure  $\sum_{n=1}^{\infty} (1-|z_n|^2)^s\delta_{z_n}$ and $F(p, p-2, s)$ via  $f''+Af=0$}

In this section, for  $0<s<1$,  $p>\max\{s, 1-s\}$,  $q>\max\{s, 1-s\}$ and a separated  sequence $\{z_n\}_{n=1}^\infty$ in $\D$ satisfying  that $\sum_{n=1}^{\infty} (1-|z_n|^2)^s\delta_{z_n}$ is an  $s$-Carleson measure, we show that there exists a function  $A$ analytic in $\mathbb{D}$ such that $|A(z)|^q(1-|z|^2)^{2q-2+s}dm(z)$ is  an $s$-Carleson measure and the equation $f''+Af=0$ admits a nontrivial solution $f \in F(p, p-2, s) \cap H^\infty$ whose zero-sequence is $\{z_n\}_{n=1}^\infty$.   From the perspective of larger  ranges of parameters $p$ and $q$ given  here, our  result improves some previous conclusions.

For  $A\in H(\D)$, it is well known that  all solutions of  the second order complex differential equation
\begin{equation}\label{41}
f''+Af=0
\end{equation}
belong to $H(\D)$.  For a  sequence $\{z_n\}$ of distinct points in  $\D$,  if  there exists    $A\in H(\D)$ such that (\ref{41}) has a solution $f$ with zeros precisely at the
points $z_n$, then  $\{z_n\}$ is said to be  a prescribed zero sequence (cf.  \cite{H1, HL}).
See  \cite{H2} for  a  historical review in this area.   In fact, by (\ref{41}), $A=-f''/f$, which gives that any $z_n$ must be a simple zero of $f$; otherwise $A$ is not analytic at $z_n$. We refer to \cite{Gr1} for some recent results on (\ref{41}) associated with Carleson measures. 

For any positive integer $n$,
recall that the Bloch space $\B$ is also equal to the set  of functions $f\in H(\D)$ such that
$$
\sup_{z\in \D} (1-|z|^2)^n|f^{(n)}(z)|<\infty.
$$
For $0\leq \alpha<\infty$, the growth space $H^\infty_\alpha$ is the set of  functions $g\in H(\D)$ satisfying
$$
\|g\|_{H^\infty_\alpha}=\sup_{z\in \D}(1-|z|^2)^\alpha |g(z)|<\infty.
$$
The following result is  Lemma 1 in  \cite{Gr}.

\begin{otherl} \label{the growth}
Let $g\in H^\infty_\alpha$ for $0\leq \alpha<\infty$, and let $0<\gamma<1$. If $g(z_0)=0$ for some $z_0\in \D$, then there exists a positive constant
$C=C(\alpha, \gamma)$ such that
$$
|g(z)|\leq \frac{C \|g\|_{H^\infty_\alpha}\rho(z, z_0)}{(1-|z_0|^2)^\alpha}, \ \ z\in \Delta(z_0, \gamma).
$$
\end{otherl}

Proposition \ref{4main1}  below   is of interest only if $H^\infty \nsubseteq X$.  Its  proof is based on a theoretical abstraction from a  well-known method in the filed of second order complex differential equations  with  prescribed zero sequences  (cf. \cite[p. 47]{H2} or \cite[pp. 301-302]{Gr}).

\begin{prop} \label{4main1}
Let $X$ be a vector space of analytic functions in $\D$  satisfying  conditions (a), (b) and (c) below.
\begin{enumerate}
  \item [(a)] Suppose  $B$ is a Blaschke product   associated with a sequence $\{a_k\}_{k=1}^\infty$ in $\D$  satisfying   $\inf_{m}\prod_{n\neq m} \rho(a_n, a_m)\geq \gamma$ for some
  $\gamma \in (0, 1)$ and $I$ is a Blaschke product  associated with a sequence $\{b_k\}_{k=1}^\infty$ in $\D$  such that  $\rho(a_k, b_k)\leq s$, $k=1, 2, \cdots$,  for some constant $s\in (0, \gamma/2)$.
  If   $B\in X$, then  $I$  also belongs to $X$.
  \item [(b)] If both  $B$ and $I$ are interpolating Blaschke products in $X$ and $c$ is a complex  constant, then the function $\exp(cBI)$ belongs to $X$.
  \item [(c)] If both $f$ and $g$ are in $H^\infty \cap X$, then $fg\in  H^\infty \cap X$.
\end{enumerate}
If $J$ is an interpolating Blaschke product associated with a sequence $\{z_k\}_{k=1}^\infty$ in $\D$ and $J\in X$, then there exists a function  $A$ analytic in $\mathbb{D}$ such that
$\sup_{z\in \D} (1-|z|^2)^2|A(z)|<\infty$ and the equation $f''+Af=0$ admits a nontrivial solution $f \in H^\infty \cap X$ whose zero-sequence is $\{z_k\}_{k=1}^\infty$.
\end{prop}
\begin{proof} Since $J$ is an interpolating Blaschke product associated with  $\{z_k\}_{k=1}^\infty$,  there is  a positive constant  $\gamma_1$ such that
$$
\inf_{n} \prod_{m \neq n}\left|\frac{z_m-z_n}{1-\overline{z_m}z_n}\right|=\inf_{n} (1-|z_n|^2)|J'(z_n)|\geq \gamma_1.
$$
Note that any bounded analytic function is a Bloch function. Then
$$
\sup_{z\in \D}(1-|z|^2)^2|J''(z)|<\infty.
$$
Hence,
$$
\sup_{n}\frac{|J''(z_n)|}{|J'(z_n)|^2}<\infty.
$$
Set
$$
\zeta_n=-\frac{J''(z_n)}{2(J'(z_n))^2}, \  \ \ \ n =1, 2, \cdots
$$
 Then $\{\zeta_n\}$ is a bounded sequence.
By Earl's constructive proof of Carleson's  interpolating theorem for $H^\infty$(see \cite{E}), there exists a complex constant $c$ depending only on $\gamma_1$ and
a Blaschke product
$$
J_1(z)= \prod_{k=1}^\infty \f{|\eta_k|}{\eta_k}\f{\eta_k-z}{1-\overline{\eta_k}z}
$$
such that
$$
c (\sup_k |\zeta_k|)  J_1(z_n)=\zeta_n, \ \ n =1, 2, \cdots.
$$
Moreover, the zeros $\{\eta_n\}$ of the Blaschke product $J_1(z)$ can be chosen to satisfy
$$
\rho(z_n, \eta_n)\leq \frac{\gamma_1}{3},  \ \ n=1, 2, \cdots.
$$
By   $J\in X$ and condition (a), we see  that $J_1\in X$.

If  $n\not=k$, then
\begin{eqnarray*}
\rho(\eta_n, \eta_k)&\geq& \rho(z_n, \eta_k)-\rho(\eta_n, z_n)\\
&\geq& \rho(z_n, z_k)-\rho(z_k, \eta_k)-\rho(\eta_n, z_n)\\
&\geq& \gamma_1-\frac{2}{3}\gamma_1.
\end{eqnarray*}
Thus $\{\eta_n\}$  is  separated. Since $\{z_k\}_{k=1}^\infty$  is uniformly separated,  $\sum_{n} (1-|z_n|^2) \delta_{z_n}$ is a 1-Carleson measure, namely,
$$
\sup_{a\in \D} \sum_n \frac{(1-|z_n|^2)(1-|a|^2)}{|1-\overline{a}z_n|^2}<\infty.
$$
Note that  $\rho(z_n, \eta_n)\leq \gamma_1/3$ for all $n$. It is known  (cf.  \cite[p. 69]{Zhu} and \cite[Lemma 4.30]{Zhu})   that
$$
1-|z_n|^2\thickapprox 1-|\eta_n|^2, \ \ |1-\overline{a}z_n|\thickapprox |1-\overline{a}\eta_n|,
$$
for all $n$ and $a\in \D$. Consequently,
$$
\sup_{a\in \D} \sum_n \frac{(1-|\eta_n|^2)(1-|a|^2)}{|1-\overline{a}\eta_n|^2}<\infty.
$$
Then  $\{\eta_n\}_{n=1}^\infty$  is also uniformly separated and hence $J_1$ is also an interpolating Blaschke product.

Set $f=Je^h$, where $h=c (\sup_k |\zeta_k|) J_1 J$. Then $f$ satisfies the equation $f''+Af=0$, where  $A\in H(\D)$ and
$$
A=-\frac{f''}{f}=-\frac{J''+2J'h'}{J}-(h')^2-h''.
$$
Condition (b) yields that $e^h\in X$. Clearly, $e^h\in H^\infty$.
We get $f\in H^\infty \cap X$ from condition (c). It is also clear that the zero set of $f$ is $\{z_k\}_{k=1}^\infty$.

Since $J$ and $h$ are in $H^\infty$ which is a subset of $\B$,  $(1-|z|^2)^2|h'(z)|^2$ and $(1-|z|^2)^2|h''(z)|$ are uniformly bounded on $\D$.
Write
$$
\Omega=\D \setminus \left( \bigcup_{n=1}^\infty \Delta(z_{n}, \frac{\gamma_1}{4}) \right).
$$
By Lemma \ref{B away from 0},
$$
\sup_{z\in \Omega} (1-|z|^2)^2 \left|\frac{J''(z)+2J'(z)h'(z)}{J(z)}\right|<\infty.
$$
Note that $J''(z_n)+2J'(z_n)h'(z_n)=0$  for all $n$. It follows from  Lemma \ref{|B|} and Lemma \ref{the growth} that
$$
\sup_{z\in \D \setminus \Omega} (1-|z|^2)^2 \left|\frac{J''(z)+2J'(z)h'(z)}{J(z)}\right|<\infty.
$$
Hence $\sup_{z\in \D} (1-|z|^2)^2|A(z)|<\infty$. We finish the proof.
\end{proof}

\vspace{0.1truecm}
\noindent {\bf  Remark 1.}\ \  It is easy to see that    (a) and (b) in Proposition \ref{4main1} can be replaced by (c) and (d) below simultaneously.
\begin{enumerate}
  \item [(c)] If an interpolating Blaschke product  $B$ associated with a sequence $\{a_k\}_{k=1}^\infty$ belongs  to $X$, then the Blaschke product
  $I$ associated with a sequence $\{b_k\}_{k=1}^\infty$ also belongs to $X$, where $\rho(a_k, b_k)<s$, $k=1, 2, \cdots$,  for some constant $s\in (0, 1)$.
  \item [(d)] Suppose   $B$ is an interpolating Blaschke product in $X$,  $I$ is a Carleson-Newman Blaschke product  in $X$ and $c$ is a complex  constant. Then the function $\exp(cBI)$ belongs to $X$.
\end{enumerate}

\vspace{0.1truecm}
\noindent {\bf  Remark 2.}\ \  For $X\subseteq H(\D)$, denote by $M(X)$ the set of  multipliers on $X$; that is,
$$
M(X)= \{f\in X:\ \ fg\in X \  \text{for all}\  g\in X \}.
$$
From \cite[Lemma 11]{DRS}, if $X$ is a Banach space of analytic functions on which point evaluations are  bounded, then $M(X)\subseteq H^\infty$.
Replacing $X$ in Proposition \ref{4main1} by $M(X)$, we get a corresponding result on $M(X)$ immediately.

\vspace{0.1truecm}
\noindent {\bf  Remark 3.}\ \
 For a special  space $X$ satisfying  assumptions in Proposition \ref{4main1},  it is interesting  to consider further the pointwise growth condition or integrated growth condition of  the function $A$  in Proposition \ref{4main1}.

In 2019  J. Gr\"ohn \cite{Gr} gave several  interesting  results for solutions of (\ref{41}) having prescribed  zeros in $\D$. In particular, he obtained the following result.

\begin{otherth}\label{JGranne}
Let $0<s\leq 1$. If $\Lambda\subset \mathbb{D}$ is a separated sequence such that $\sum_{z_n \in \Lambda} (1-|z_n|)^s \delta_{z_n}$ is an $s$-Carleson measure, then there exists a function  $A$ analytic in $\mathbb{D}$ such that $|A(z)|^2(1-|z|^2)^{2+s}dm(z)$ is  an $s$-Carleson measure and the equation $f''+Af=0$ admits a nontrivial solution $f \in \Q_s \cap H^\infty$ whose zero-sequence is $\Lambda$.
\end{otherth}

For $0<p_1<p_2<\infty$ and $0<s\leq 1$ with $p_1+s>1$, it is known that $F(p_1, p_1-2, s)\subsetneqq  F(p_2, p_2-2, s)$. For $s>1$, all nontrivial $F(p, p-2, s)$ spaces are equal to the Bloch space.
Suppose $0<p<\infty$ and $0<s<\infty$ with $2p+s>1$. For $A\in H(\D)$,  $|A(z)|^p(1-|z|^2)^{2p-2+s}dm(z)$ is  an $s$-Carleson measure if and only if
$$
\sup_{w\in \D}\ind |A(z)|^p(1-|z|^2)^{2p-2} (1-|\sigma_w(z)|^2)^s dm(z)<\infty.
$$
Denote by $N_{p, s}$ the space of analytic functions $A$ satisfying the formula above. Clearly, if $0<p<\infty$ and $0<s<\infty$ with $2p+s\leq 1$, then $N_{p, s}$ contains only constant functions.
For $0<p<\infty$ and $s>1$, $N_{p, s}$ is equal to the Bloch type space $\B^3$ consisting of functions $f\in H(\D)$ with
$$
\sup_{z\in \D}(1-|z|^2)^3|f'(z)|<\infty.
$$
From \cite[Remark 3]{Ye}, if $0<p_1<p_2<\infty$ and $0<s\leq 1$ with $2p_1+s>1$, then $N_{p_1, s}\subsetneqq N_{p_2, s}$.
Thus the following theorem from \cite{Ye}  is a proper generalization of  the case of $0<s<1$ in Theorem \ref{JGranne}.

\begin{otherth} \label{Ye}
Let $0<s<1<p<\infty$ and $1<q<\infty$.  If $\Lambda\subset \mathbb{D}$ is a separated sequence such that $\sum_{z_n \in \Lambda} (1-|z_n|)^s \delta_{z_n}$ is an $s$-Carleson measure, then there exists a function  $A$ analytic in $\mathbb{D}$ such that $|A(z)|^q(1-|z|^2)^{2q-2+s}dm(z)$ is  an $s$-Carleson measure and the equation $f''+Af=0$ admits a nontrivial solution $f \in F(p, p-2, s) \cap H^\infty$ whose zero-sequence is $\Lambda$.
\end{otherth}

Theorems \ref{Bao-Fp1}  below  is a  sequel to Theorem \ref{JGranne} and Theorem  \ref{Ye}.  Form the explanations after Theorem \ref{JGranne}, we see that Theorem  \ref{Bao-Fp1} below   strengthens Theorem \ref{Ye} and
the case of $0<s<1$ in Theorem \ref{JGranne}.

\begin{thm}\label{Bao-Fp1}
Suppose $0<s<1$,  $p>\max\{s, 1-s\}$ and  $q>\max\{s, 1-s\}$.
If $\Lambda\subset \mathbb{D}$ is a separated sequence such that $\sum_{z_n \in \Lambda} (1-|z_n|)^s \delta_{z_n}$ is an $s$-Carleson measure, then there  exists a function     $A$ in $H(\D)$ such that $|A(z)|^q(1-|z|^2)^{2q-2+s}dm(z)$
 is  an $s$-Carleson measure and the equation $f''+Af=0$ admits a nontrivial solution $f \in F(p, p-2, s) \cap H^\infty$ whose zero-sequence is $\Lambda$.
\end{thm}
\begin{proof}

Note that $0<s<1$ and  $p>\max\{s, 1-s\}$. We first show that  (a), (b) and (c) in Proposition  \ref{4main1} hold by taking $X= F(p, p-2, s)$. Using Theorem \ref{1F-Inner} and checking the proof of Proposition \ref{4main1}, we get that (a) in Proposition  \ref{4main1} holds when $X= F(p, p-2, s)$.  Clearly, if $X=F(p, p-2, s)$, then  (c) in  Proposition \ref{4main1}  also holds.
Suppose $B$ and $I$ are interpolating Blaschke products in $F(p, p-2, s)$ and $c$ is a complex constant. Then
\begin{eqnarray*}
&~&\sup_{a\in \D} \ind \left|(e^{cB(z)I(z)})'\right|^p (1-|z|^2)^{p-2}(1-|\sigma_a(z)|^2)^sdA(z)\\
&\lesssim&  \|I\|^p_{F(p, p-2, s)} \sup_{z\in \D}\left|e^{cB(z)I(z)}\right|^p+  \|B\|^p_{F(p, p-2, s)} \sup_{z\in \D}\left|e^{cB(z)I(z)}\right|^p
< \infty,
\end{eqnarray*}
which gives  $e^{cBI}\in F(p, p-2, s)$. Thus (b) in  Proposition \ref{4main1} also  holds when  $X=F(p, p-2, s)$.

For  convenience, we write $\Lambda=\{z_n\}_{n=1}^\infty$. Denote by $J$ the Blaschke product  associated with the  sequence $\{z_n\}_{n=1}^\infty$.
Since $\{z_n\}$ is  separated and $\sum_{z_n \in \Lambda} (1-|z_n|)^s \delta_{z_n}$ is an $s$-Carleson measure, $J$ is an interpolating Blaschke product and it follows from Theorem \ref{1F-Inner} that $J\in F(p, p-2, s)$.
By Proposition \ref{4main1}, there exists a function  $A$ analytic in $\mathbb{D}$ such that the equation $f''+Af=0$ admits a nontrivial solution $f \in H^\infty \cap F(p, p-2, s)$ whose zero-sequence is $\{z_n\}_{n=1}^\infty$.
As shown in the proof of Proposition \ref{4main1},
$$
A=-\frac{J''+2J'h'}{J}-(h')^2-h'', \ \ h=C J_1 J,
$$
where $C$ is a complex constant and $J_1$ is a Blaschke product whose zero-sequence $\{\eta_n\}_{n=1}^\infty$ satisfies
$$
\rho(z_n, \eta_n)\leq \frac{\gamma}{3},  \ \ n=1, 2, \cdots
$$
for some $\gamma \in (0, 1)$.
For $q>\max\{s, 1-s\}$,   repeating  the arguments of the proof of Theorem 1.1 in \cite{Ye} (i.e. Theorem \ref{Ye} stated in this paper),  we see that $|A(z)|^q(1-|z|^2)^{2q-2+s}dm(z)$ is  an $s$-Carleson measure.
We finish the proof.
\end{proof}

Let $0<s<1$, $q>\max\{s, 1-s\}$ and $A\in H(\D)$. Suppose  $|A(z)|^q(1-|z|^2)^{2q-2+s}dm(z)$ is  an $s$-Carleson measure. By (\ref{SCM}) and the subharmonicity of $|A|^q$, we deduce
\begin{align*}
\infty &> \sup_{w\in \D}  \int_{\Delta (w, 1/2)}  \left(\frac{1-|w|^2}{|1-\overline{w}z|^2}\right)^s |A(z)|^q(1-|z|^2)^{2q-2+s}dm(z)\\
&\gtrsim \sup_{w\in \D} |A(w)|^q(1-|w|^2)^{2q},
\end{align*}
which implies  the condition of $A$ appeared in Proposition \ref{4main1}. In other words, the condition of $A$ in Theorem \ref{Bao-Fp1} is stronger than that in Proposition \ref{4main1}.

It is known (cf. \cite{PR, Zhao}) that $H^\infty \subseteq F(p, p-2, 1)$ when $2\leq p<\infty$, but  for $0<p<2$,   $H^\infty \nsubseteq F(p, p-2, 1)$
and $F(p, p-2, 1)\nsubseteq H^\infty$.  It is natural to consider  the case of $s=1$ in Theorem \ref{JGranne} via  replacing $\Q_1\cap H^\infty$ (i.e. $H^\infty$) by  $F(p, p-2, 1)\cap H^\infty$ for  $0<p<2$.  For this purpose, one should  characterize  interpolating Blaschke products  $B$ in $F(p, p-2, 1)$ for $0<p<2$ via the distribution of zeros of $B$. It is still  open  to find this characterization (cf. \cite[p.755]{PR}).

\vspace{0.1truecm}
\noindent {\bf Data Availability. }

All data generated or analyzed during this study are included in this article and in its bibliography.

\vspace{0.1truecm}
\noindent {\bf Conflict of Interest. }

The authors declared that they have no conflict of interest.

\end{document}